\documentclass[reqno]{amsart}

\usepackage{amscd, amssymb, amsthm, mathrsfs, amsmath, amsrefs}
\usepackage[all]{xy}

\numberwithin{equation}{section}

\theoremstyle{plain}
\newtheorem{thm}[equation]{Theorem}
\newtheorem{prop}[equation]{Proposition}
\newtheorem{coro}[equation]{Corollary}
\newtheorem{lem}[equation]{Lemma}
\newcommand{\comment}[1]{}

\theoremstyle{definition}
\newtheorem{defi}[equation]{Definition}

\theoremstyle{remark}

\newtheorem{rem}[equation]{Remark}

\theoremstyle{definition}

\numberwithin{equation}{section}

\usepackage{tikz}
\usetikzlibrary{matrix,arrows}

\newcommand{\sch}{\mathbf{Sch}_k}

\newcommand{\spt}{\mathbf{Spt}}
\newcommand{\ass}{\mathbf{Rng}}

\newcommand{\dgcatq}{\mathbf{dgCat}_{\Q}}
\newcommand{\dgcat}{\mathbf{dgCat}_k}
\newcommand{\grpd}{\mathbf{Grpd}}
\newcommand{\Or}{\mathbf{Or}}
\newcommand{\sset}{\mathbf{SSet}}
\newcommand{\sets}{\mathbf{Sets}}
\newcommand{\weq}{\overset{\sim}\longrightarrow}
\def\fibeq{\overset\sim\twoheadrightarrow}
\def\cat{\mathbf{Cat}}

\newcommand{\dgm}{\mathbf{C}}
\newcommand{\mdg}{\mathbf{C_{\mathrm{dg}}}}

\newcommand{\perf}{\mathrm{Perf}}
\newcommand{\Dperf}{\mathbf{D}_{\mathrm{perf}}}
\newcommand{\Perf}{\mathbf{Perf}}
\newcommand{\ac}{\mathrm{Ac}}

\newcommand{\cHom}{\mathcal{H}\hspace{-0.1em}\mathit{om}}
\newcommand{\cal}[1]{\mathcal{#1}}
\newcommand{\Ho}{\mathrm{Ho}}
\newcommand{\otk}{\otimes_k}
\newcommand{\otq}{\otimes_{\Q}}
\newcommand{\Oo}{\mathcal{O}}
\newcommand{\Grpd}{\mathcal{G}}
\newcommand{\Fam}{\mathcal{F}}
\newcommand{\Fib}{\mathscr{F}}
\newcommand{\fS}{\mathfrak{S}}
\newcommand{\fT}{\mathfrak{T}}
\newcommand{\car}{\mathrm{char}}
\def\cA{\cal A}
\def\cG{\cal G}

\def\cE{\cal E}
\def\cF{\cal F}
\def\cT{\cal T}

\def\fC{\mathfrak{C}}
\newcommand{\op}{\mathrm{op}}
\newcommand{\ob}{\mathrm{ob}}
\newcommand{\red}{\mathrm{red}}
\newcommand{\zar}{\mathrm{Zar}}
\newcommand{\ch}{\mathrm{ch}}
\newcommand{\cdh}{\mathrm{cdh}}
\newcommand{\aff}{\mathrm{aff}}
\renewcommand{\inf}{\mathrm{inf}}
\newcommand{\map}{\mathrm{map}}

\newcommand{\CYC}{\mathcal{C}yc}
\newcommand{\cyc}{\CYC}
\newcommand{\vcyc}{\cal{V}cyc}
\newcommand{\Z}{\mathbb{Z}}
\newcommand{\Q}{\mathbb{Q}}

\DeclareMathOperator*{\hocolim}{hocolim}

\DeclareMathOperator{\hofiber}{hofiber}
\newcommand{\hofi}{\hofiber}
\DeclareMathOperator{\Spec}{Spec}
\DeclareMathOperator{\pretr}{\mathrm{pretr}}

\begin{document}
\title{Singular coefficients in the $K$-theoretic Farrell-Jones conjecture}
\author{Guillermo Corti\~nas}
\email{gcorti@dm.uba.ar}\urladdr{http://mate.dm.uba.ar/\~{}gcorti}
\author{Emanuel Rodr\'\i guez Cirone}
\email{ercirone@dm.uba.ar}
\address{Dep. Matem\'atica-IMAS, FCEyN-UBA\\ Ciudad Universitaria Pab 1\\
1428 Buenos Aires\\ Argentina}
\thanks{The first author was partially supported by MTM2012-36917-C03-02. Both authors
were supported by CONICET, and partially supported by
grants UBACyT 20020100100386 and PIP 11220110100800.}
\begin{abstract}
Let $G$ be a group and let $k$ be a field of characteristic zero. We prove that if the Farrell-Jones conjecture for the $K$-theory of $R[G]$ is satisfied for every smooth $k$-algebra $R$, then it is also satisfied for every commutative $k$-algebra $R$.
\end{abstract}
\maketitle

\section{Introduction}

Let $G$ be a group; a \emph{family} of subgroups of $G$ is a
nonempty family $\cF$ closed under conjugation and under taking
subgroups. A \emph{$G$-space} is a simplicial set together with a $G$-action. 
If $\cF$ is a family of subgroups of $G$ and $f:X\to Y$ is an equivariant
map of $G$-spaces, then we say that $f$ is an \emph{$\cF$-equivalence} (resp. an \emph{$\cF$-fibration}) if the map between fixed
point sets
\[
f:X^H\to Y^H
\]
is a weak equivalence (resp. a fibration) for every $H\in \cF$. A
$G$-space $X$ is called a {\em $(G,\cF)$-complex} if the
stabilizer of every simplex of $X$ is in $\cF$. The category of
$G$-spaces can be equipped with a closed model structure
where the weak equivalences (resp. the fibrations) are the $\cF$-equivalences (resp. the $\cF$-fibrations), 
(see \cite{corel}*{\S1}). The $(G,\cF)$-complexes
are the cofibrant objects in this model structure. By a general
construction of Davis and L\"uck (see \cite{dl}) any functor $E$
from the category $\Z$-$\cat$ of small $\Z$-linear categories to the
category $\spt$ of spectra which sends category equivalences to
weak equivalences of spectra gives rise to an equivariant homology theory
of $G$-spaces $X\mapsto H^G(X,E(R))$ for each unital ring $R$. If $H\subset G$ is a subgroup, 
then
\begin{equation}\label{intro:cross}
H_*^G(G/H,E(R))=E_*(R[H])
\end{equation}
is just $E_*$ evaluated at the group ring.
The \emph{strong isomorphism conjecture} for the cuadruple $(G,\cF,E,R)$
asserts that if $f:X\to Y$ is an $\cF$-equivalence, then the induced map
$H^G_*(f,E(R))$ is an isomorphism. The \emph{isomorphism
conjecture} for the quadruple $(G,\cF,E,R)$ asserts that if
$\cE(G,\cF)\fibeq pt$ is a $(G,\cF)$-cofibrant replacement of the
point, then the induced map
\begin{equation}\label{intro:assem}
H_*^G(\cE(G,\cF),E(R))\to E_*(R[G])
\end{equation}
--called \emph{assembly map}-- is an isomorphism. 

A group is called \emph{virtually cyclic} if it contains a cyclic group of finite index. The $K$-theoretic \emph{Farrell-Jones conjecture} for a group $G$ with coefficients in a ring $R$ is the isomorphism conjecture for the quadruple $(G,\vcyc, K,R)$; here $\vcyc$ is the family of virtually cyclic subgroups of $G$.

Our main result is the following.

\begin{thm}\label{intro:main}
Let $\Fam$ be a family of subgroups of $G$ that contains all the cyclic subgroups. Let $k$ be a field of characteristic zero and let $f:X\to Y$ be a $(G,\Fam)$-equivalence. Suppose that $H^G(f,K(R))$ is a weak equivalence for every commutative smooth $k$-algebra $R$. Then $H^G(f,K(R))$ is a weak equivalence for every commutative $k$-algebra $R$. In particular, if the (strong) isomorphism conjecture for $(G,\Fam, K, R)$ holds for every commutative smooth $k$-algebra $R$, then it holds for every commutative $k$-algebra $R$.
\end{thm}

\begin{coro}\label{intro:cormain}
Let $G$ be a group. If $G$ satisfies the $K$-theoretic Farrell-Jones conjecture with coefficients in every commutative smooth $\Q$-algebra $R$ then it also satisfies the Farrell-Jones conjecture with coefficients in any commutative $\Q$-algebra.
\end{coro}

Next we sketch the structure of the proof of Theorem \ref{intro:main}. First observe that any commutative $k$-algebra is a filtering colimit of subalgebras of finite type. Hence it suffices to prove the theorem for $R$ of finite type over $k$, since $K$-theory commutes with filtering colimits. Recall that the opposite of the category of commutative $k$-algebras of finite type embeds as the full subcategory of affine schemes inside the category $\sch$ of separated $k$-schemes of finite type. The $K$-theory of a scheme $\fS$ can be  defined as the $K$-theory of a certain dg-category $\Perf_\fS$.  We extend the definition of equivariant homology so that we can take coefficients in a scheme $\fS\in\sch$. It is characterized by
\[
H^G_*(G/H,K(\fS))=K_*(\Perf_\fS\otimes \Z[H]).
\]
Here $\otimes$ is the tensor product of dg-categories; $\Z[H]$ is considered as a dg-category with trivial grading and zero differential. When $\fS$ is affine, we recover the usual definition of equivariant $K$-homology; we have
\[
K_*(\Perf_{\Spec R}\otimes\Z[H])=K_*(R[H]).
\]

Our construction of equivariant $K$-homology with scheme coefficients gives a functorial spectrum $H^G(X,K(\fS))$, covariant in the first variable and contravariant in the second. 
Hence for each $G$-space $X$, we have a presheaf of spectra on $\sch$
\[
\fS\mapsto H^G(X,K(\fS)).
\]
Any equivariant map $f:X\to Y$ induces a morphism of presheaves $H^G(f,K(\fS))$. Our hypothesis says that for $f$ as in the theorem and $\fS$ smooth affine, the natural map $H^G(f,K(\fS))$ is a weak equivalence; we would be done if we could prove that $H^G(f,K(\fS))$ is a weak equivalence for all $\fS\in\sch$. There is a Grothendieck topology on $\sch$, Voevodsky's $cdh$-topology, with respect to which every scheme $\fS$ is locally smooth (here we use that $\car(k)=0$). Hence our hypothesis says that $H^G(f,K(\fS))$ is a local weak equivalence; we want to show that it is a global weak equivalence. To do this we need some results from the theory of presheaves of spectra and Grothendieck topologies. Given a category $\cal{C}$ with a Grothendieck topology $t$ there is a model category structure on the category of presheaves of spectra on $\cal{C}$ where the weak equivalences are the $t$-local weak equivalences, and is such that any weak equivalence between fibrant presheaves is a global weak equivalence. In particular any $t$-local weak equivalence $\mathscr{S}\to\mathscr{T}$ induces a global weak equivalence $\mathscr{S}_t\to\mathscr{T}_t$ between fibrant replacements. Thus in our case we have that $H^G(f,K(\fS))_\cdh$ is a global equivalence. Let 
\[
\Fib H^G(X,K(\fS))=\hofi(H^G(X,K(\fS))\to H^G(X,K(\fS))_\cdh).
\] 
We have a map of fibration sequences
\begin{equation}\label{intro:mapseq}
\xymatrix{
\Fib H^G(X,K(\fS))\ar[r]\ar[d]& H^G(X,K(\fS))\ar[d]\ar[r]&H^G(X,K(\fS))_\cdh\ar[d]^\wr\\
\Fib H^G(Y,K(\fS))\ar[r]& H^G(Y,K(\fS))\ar[r]&H^G(Y,K(\fS))_\cdh\\ 
}
\end{equation}

We know that the rightmost vertical map is a weak equivalence, and we want to show that the same is true of the vertical map in the middle. By the five lemma, it suffices to show that the leftmost vertical map is a weak equivalence. There is a similar sequence with cyclic homology substituted for $K$-theory, and we prove in Theorem \ref{thm:hgfc} that for every $G$-space $X$ and every $\fS\in\sch$ there is a weak equivalence
\begin{equation}\label{intro:mapfkfhc}
\Fib H^G(X,K(\fS))\weq \Fib H^G(X,HC(\fS))[-1].
\end{equation}
Here the cyclic homology $HC$ is taken over $\Z$.

We prove in Proposition \ref{prop:ichc} that if $\Fam$ contains all the cyclic subgroups of $G$ and $f:X\to Y$ is a $\Fam$-equivalence of $G$-spaces, then the induced map
\begin{equation}\label{intro:ichc}
\Fib H^G(X,HC(\fS))\weq \Fib H^G(Y,HC(\fS))
\end{equation}
is a weak equivalence for all $\fS\in\sch$. This concludes the proof. 

The rest of this paper is organized as follows. At the beginning of Section \ref{sec:prelis} we recall some basic notions about 
presheaves of spectra which we shall need. These include the notion of descent with respect to the Zariski, Nisnevich, and $cdh$ topologies. Each of these topologies is generated by a class of cartesian squares in $\sch$
\begin{equation}\label{intro:square}
\xymatrix{
\tilde{\fT}\ar[d]^{q}\ar[r]^{j} & \tilde{\fS}\ar[d]^{p} \\
\fT\ar[r]^{i} & \fS
}
\end{equation}
closed under isomorphisms; such a class is called a $cd$-structure. A presheaf $E$ of spectra satisfies \emph{descent} with respect to the square \eqref{intro:square} if it sends it to a homotopy cartesian diagram of spectra. If $P$ is a $cd$-structure, we say that $E$ satisfies descent with respect to $P$ if it satisfies descent for every square in $P$. If $P$ is any of the $cd$-structures considered in this paper, descent with respect to $P$ is equivalent to descent with respect to the topology $t_P$ generated by $P$. Towards the end of the section we recall some basic notions from the theory of $dg$-categories, including the definition of the pretriangulated $dg$-category $\Perf_\fS$ associated to a scheme $\fS$. Section \ref{sec:desmor} is concerned with descent properties of functors from $dg$-categories to spectra which are Morita invariant and localizing (such as, for example, $K$-theory and cyclic homology and its variants). We prove in Theorem \ref{thm:regularblowups} that if $E$ is such a functor and $\cA$ is a $dg$-category over $k$, then $\fS\mapsto E(\fS\otimes_k\cA)$ satisfies descent with respect to those cartesian squares \eqref{intro:square} such that $i$ is a regular closed immersion of pure codimension $d$ and $p$ is the blow-up along $i$. In Section \ref{sec:eqhom} we introduce equivariant homology with scheme coefficients. We show in Subsection \ref{subsec:equikhc} that any Morita invariant functor $E$ from $dg$-categories to spectra gives rise to an equivariant homology theory of $G$-spaces $H^G(X,E(\fS))$ with coefficients in $\fS\in\sch$ such that $H^G(G/H,E(\fS))\weq E(\Perf_{\fS}\otimes\Z[H])$. In Subsection \ref{subsec:equiche} we construct a functorial equivariant Chern character
$H^G(X,K(\fS))\to H^G(X,HN(\fS))$ from equivariant $K$-theory to equivariant negative cyclic homology. We write $H^G(X,K^{\inf}(\fS))$ for the homotopy fiber of this character. In Section \ref{sec:descequi} we use Theorem \ref{thm:regularblowups} and a result from \cite{chsw} which we recall in Theorem \ref{thm:chsw}, to prove in Proposition \ref{prop:cdh} that if $k$ is a field of characteristic zero, then the presheaves $H^G(X,K^{\inf}(?))$ and $H^G(X,HP(?))$ satisfy $cdh$-descent on $\sch$; here $HP$ is periodic cyclic homology. Then we use the latter result to prove \eqref{intro:mapfkfhc} (Theorem \ref{thm:hgfc}). In Section \ref{sec:ic} we prove \eqref{intro:ichc} (Proposition \ref{prop:ichc}) and Theorem \ref{intro:main} (Theorem \ref{thm:main}).

\begin{rem}
By combining our Corollary \ref{intro:cormain} with results recently announced by Carlsson-Goldfarb \cite{cargo}, one could conclude that the Farrell-Jones conjecture in $K$-theory holds with coefficients in any commutative $\Q$-algebra, for $G$ a torsion-free, geometrically finite group of finite asymptotic dimension.
\end{rem}

\section{Preliminaries}\label{sec:prelis}
\numberwithin{equation}{subsection}
\subsection{cd-Structures}\label{subsec:cdst}

Let $k$ be a field. Let $\sch$ be the category of separated $k$-schemes of finite type and let
\begin{equation}\label{square}
\xymatrix{\tilde{\fT}\ar[d]^{q}\ar[r]^{j} & \tilde{\fS}\ar[d]^{p} \\
\fT\ar[r]^{i} & \fS \\}
\end{equation}
be a cartesian square in $\sch$. We say that the cartesian square \eqref{square} is \emph{Nisnevich} if $i$ is an open immersion and $p$ is an \'{e}tale morphism that induces an isomorphism between the \emph{reduced} schemes
\begin{equation}\label{map:nisnecond}
(\tilde{\fS}-\tilde{\fT})_{\red}\overset{\simeq}\longrightarrow(\fS-\fT)_{\red}.
\end{equation} 
A \emph{Zariski} square is a Nisnevich square such that both $i$ and $p$ are open immersions; in this case condition \eqref{map:nisnecond} means that $\fS=i(\fT)\cup p(\tilde{\fS})$. We say that the cartesian square \eqref{square} is a \emph{regular blow-up} if $i$ is a regular closed embedding of pure codimension $d$ and $\tilde{\fS}$ is the blowup of $\fS$ along $\fT$. In general a cartesian square \eqref{square} is an \emph{abstract blow-up} if $i$ is a closed immersion and $p$ is a proper morphism that induces an isomorphism $(\tilde{\fS}-\tilde{\fT})\to (\fS-\fT)$.  A \emph{cd-structure} on a small category $\cal{C}$ is a set $P$ of commutative squares in $\cal{C}$ which is closed under isomorphisms. A category $\cal{C}$ with a cd-structure $P$ has an associated Grothendieck topology $t_P$ (see \cite{voeho}*{Section 2}). The \emph{Zariski} and the \emph{Nisnevich} topologies in $\sch$ are the topologies generated by the Zariski and the Nisnevich squares, respectively. The \emph{$cdh$ topology} is the topology on $\sch$ generated by the \emph{combined $cd$ structure} which consists of the Nisnevich squares and the abstract blow-up squares. 

\subsection{Presheaves of spectra and descent}\label{subsec:prsh}

Let $\cal{C}$ be a small category. The category $\spt^{\cal{C}^{\op}}$ of presheaves of spectra has a model structure, called the \emph{global injective model structure}, in which weak equivalences and cofibrations are defined objectwise and fibrations are defined by the right lifting property. If $\cal{C}$ is equipped with a Grothendieck topology $t$ we can endow $\spt^{\cal{C}^{\op}}$ with a different model structure, as we proceed to explain. Let $a_t$ denote the associated sheaf functor from the category of presheaves of abelian groups on $\cal{C}$ to the category of sheaves of abelian groups on $(\cal{C}, t )$. A \emph{$t$-local weak equivalence} is a morphism $E\to F$ that induces an isomorphism $a_t\pi_*E\to a_t\pi_*F$ on the associated sheaves of stable homotopy groups. The \emph{$t$-local injective fibrations} are defined by the right lifting property with respect to those morphisms which are (objectwise) cofibrations and $t$-local weak equivalences. The classes of (objectwise) cofibrations, $t$-local injective fibrations and $t$-local weak equivalences satisfy the axioms for a model structure on $\spt^{\cal{C}^{\op}}$ \cite{jarget}*{Theorem 3.24} which we will call the \emph{$t$-local injective model structure}.

Let $E$ be a presheaf of spectra on $\cal{C}$ and let \eqref{square} be a diagram in $\cal{C}$. We say that $E$ has \emph{descent} for the square \eqref{square} if the diagram
\begin{equation}\label{square2}
\xymatrix{E(\tilde{\fT}) & E(\tilde{\fS})\ar[l] \\
E(\fT)\ar[u] & E(\fS)\ar[u]\ar[l] \\}
\end{equation}
is homotopy cartesian. Given a cd-structure $P$ on $\cal{C}$ we say that $E$ \emph{has descent for $P$} if it has descent for every square in $P$. Suppose that $P$ is complete, regular and bounded in the sense of \cite{voeho}*{Definitions 2.3, 2.10, and 2.22}. For example, the Zariski, Nisnevich, and combined cd-structures on $\sch$ have all these properties (\cite{unst}*{Theorem 2.2}). Give $\cal{C}$ the topology $t=t_P$ generated by $P$ and let $E\mapsto E_t$ be a functorial fibrant replacement for the $t$-local injective model structure in $\spt^{\cal{C}^{\op}}$. Let $E$ be a presheaf of spectra on $\cal{C}$. Then $E$ has descent for $P$ if and only if the natural map $E\to E_t$ is a global weak equivalence \cite{chsw}*{Theorem 3.4}. If this is the case we say that $E$ \emph{has descent for $t$ (or $t$-descent)}.

\subsection{dg-Categories}\label{subsec:dg}

The definitions in this subsection follow those in \cite{kelicm}. A \emph{$k$-category} is a category $\cal{A}$ whose hom-sets are $k$-modules and in which composition is $k$-bilinear. A \emph{dg-category} (over $k$) is a $k$-category $\cal{A}$ whose hom-sets are dg $k$-modules and in which the morphisms
\[\cal{A}(y,z)\otimes \cal{A}(x,y)\to \cal{A}(x,z)\]
induced by composition are morphisms of dg $k$-modules; see \cite{kelicm}*{2.1} for details.

Let $\dgm(k)$ be the category of dg $k$-modules and morphisms of dg $k$-modules. We consider $\dgm(k)$ as a closed symmetric monoidal category with the usual tensor product of dg $k$-modules and the structure we proceed to describe. Let $E$ and $F$ be dg $k$-modules. The symmetry isomorphism $E\otimes F\simeq F\otimes E$ is given by $v\otimes w\leftrightarrow (-1)^{pq}w\otimes v$ for $v\in E^p$ and $w\in F^q$. The internal hom $\cHom(E,F)$ has components
\[\cHom(E,F)^p=\prod_{q\in\Z}\hom_k(E^q,F^{q+p})\]
and differential $d(f)=d_F\circ f-(-1)^pf\circ d_E$ for homogeneous $f$ of degree $p$. With these definitions, a dg-category over $k$ is the same as a $\dgm(k)$-category in the sense of \cite{kelly}*{1.2}.

There is a dg-category $\mdg(k)$ whose objects are dg $k$-modules and whose hom-sets are given by the internal hom in $\dgm(k)$ \cite{kelicm}*{2.2}.

Let $\cal{A}$ and $\cal{B}$ be dg-categories over $k$. A \emph{dg-functor} $F:\cal{A}\to\cal{B}$ is a functor such that the functions $F(x,y):\cal{A}(x,y)\to \cal{B}(F(x),F(y))$ are morphisms of dg $k$-modules. Equivalently, a dg-functor is a $\dgm(k)$-functor in the sense of \cite{kelly}*{1.2}. We will write $\dgcat$ for the category of small dg-categories over $k$ and dg-functors. There is a symmetric tensor product $\otk$ which makes $\dgcat$ into a symmetric monoidal category \cite{kelicm}*{2.3}.

Every dg-category $\cal{A}$ has an associated category $H^0(\cal{A})$ (resp. $Z^0(\cal{A})$) which has the same objects as $\cal{A}$ and whose hom-sets are given by zero cohomology modules (resp. zero cycle modules)
\[
\hom_{H^0(\cal{A})}(x,y)=H^0\cal{A}(x,y)\hspace{1em}\mbox{(resp. $\hom_{Z^0(\cal{A})}(x,y)=Z^0\cal{A}(x,y)$).}
\]
The category $H^0(\cal{A})$ is called the \emph{homotopy category} of $\cal{A}$. A dg-functor $F:\cal{A}\to\cal{B}$ is a \emph{quasi-equivalence} if it induces an equivalence of categories $H^0(F):H^0(\cal{A})\to H^0(\cal{B})$.

Let $\cal{A}$ be a dg-category. A \emph{right $\cal{A}$-module} is a dg-functor $M:\cal{A}^{\op}\to \mdg(k)$; here $\cal{A}^{\op}$ is the \emph{opposite dg-category} \cite{kelicm}*{2.2}. We define $\dgm(\cal{A})$ to be the category whose objects are the right $\cal{A}$-modules and whose hom-sets $\hom_{\dgm(\cal{A})}(M,N)$ consist of the morphisms of functors $\varphi:M\to N$ such that $\varphi_x:Mx\to Nx$ is a morphism of dg $k$-modules for every object $x$ of $\cal{A}$. A morphism $\varphi\in\hom_{\dgm(\cal{A})}(M,N)$ is a \emph{quasi-isomorphism} if the morphisms $\varphi_x:Mx\to Nx$ induce isomorphisms in cohomology for every object $x$ of $\cal{A}$. The category $\dgm(\cal{A})$ has a model structure for which the weak equivalences are the quasi-isomorphisms \cite{kelicm}*{Theorem 3.2}. The derived category $D(\cal{A})$ is by definition the category $\Ho\dgm(\cal{A})$.

A dg-functor $\cal{A}\to \cal{B}$ is a \emph{Morita equivalence} if the restriction functor $D(\cal{B})\to D(\cal{A})$ is an equivalence of categories; see \cite{kelicm}*{4.6}. For example, every quasi-equivalence is a Morita equivalence.

\subsection{Pretriangulated dg-categories}\label{subsec:pretr}

Let $\cal{A}$ be a dg-category. The category $\dgm(\cal{A})$ has a natural $\dgm(k)$-enrichment into a dg-category $\mdg(\cal{A})$ \cite{kelly}*{2.2}. There is a \emph{Yoneda dg-functor} $Y:\cal{A}\to \mdg(\cal{A})$ which is fully faithful \cite{kelly}*{2.4}. A dg-category $\cal{A}$ is \emph{pretriangulated} if the image of the functor $Z^0(Y):Z^0(\cA)\to \dgm(\cA)$ is closed under shifts and cones of morphisms of $\cal{A}$-modules; see \cite{kelicm}*{4.5}. This notion of pretriangulated dg-category is the same as Keller's notion of \emph{exact dg-category} \cite{kelex}*{2.1} and Drinfeld's notion of \emph{strongly pretriangulated dg-category} \cite{drin}*{2.4}.

The homotopy category $H^0\mdg(\cal{A})$ is a triangulated category in a natural way. The Yoneda dg-functor induces a fully faithful functor $H^0(\cal{A})\to H^0\mdg(\cal{A})$. If $\cal{A}$ is pretriangulated, the image of this functor is closed under shifts and cones. Thus, $H^0(\cal{A})$ inherits a triangulated structure from $H^0\mdg(\cal{A})$. Any dg-functor $F:\cal{A}\to\cal{B}$ between pretriangulated dg-categories induces a triangulated functor $H^0(F):H^0(\cal{A})\to H^0(\cal{B})$. Every dg-category $\cal{A}$ admits a universal dg-functor $\cal{A}\to\pretr(\cal{A})$ into a pretriangulated dg-category \cite{kelicm}*{4.5} and this dg-functor is a Morita equivalence. Moreover, the dg-functor $\pretr(\cal{A})\to\mdg(\cal{A})$ induced by the Yoneda dg-functor $\cal{A}\to\mdg(\cal{A})$ is fully faithful. For example if $R$ is a $k$-algebra considered as a dg-category with only one object, then $\mdg(R)$ is the dg-category of cochain complexes of right $R$-modules described in \cite{kelicm}*{2.2} and $\pretr(R)$ identifies with its full subcategory of strictly bounded complexes of finitely generated free $R$-modules.

Write $\cT^\sim$ for the idempotent completion \cite{balsch}*{1.5} of a triangulated category $\cT$. Let $F:\cal{A}\to\cal{B}$ be a dg-functor with $\cal{A}$ non empty. Then $F$ is a Morita equivalence if and only if $H^0(F)$ is fully faithful and $H^0(\pretr(F))^\sim$ is essentially surjective \cite{tabuadainvariants}*{5}.

\subsection{dg-Enhancement of schemes}\label{subsec:perfect}

A \emph{dg-enhancement} of a triangulated category $T$ is a pretriangulated dg-category $\cT$ such that $H^0(\cT)$ is triangle equivalent to $T$. It is possible to construct a presheaf of dg-categories $\Perf_?:\sch^{\op}\to\dgcat$ which gives a functorial dg-enhancement of the derived category $\Dperf(\fS)$ of perfect complexes on a scheme $\fS\in\sch$; see for example \cite{chsw}*{Example 2.7} and \cite{dgperf}. We proceed to describe briefly the idea of such a construction. First consider a dg-category $\perf_\fS$ whose objects are a certain class of perfect complexes of $\Oo_\fS$-modules and whose dg $k$-modules of morphisms are given by the internal hom of cochain complexes; in both references the strict perfect complexes of free $\Oo_\fS$-modules are objects of $\perf_\fS$. Then let $\ac_\fS$ be the full subcategory of $\perf_\fS$ whose objects are the acyclic complexes and define $\Perf_\fS$ to be the Drinfeld quotient $\perf_\fS/\ac_\fS$ \cite{drin}*{3.1}.

Let $R$ be a commutative $k$-algebra of finite type and let $\fS=\Spec R$. Consider $R$ as a dg-category over $k$ with only one object; there is a dg-functor $R\to\perf_\fS$ which sends the unique object of $R$ to the $\Oo_\fS$-module $\Oo_\fS$ concentrated in degree zero. Composing this with the dg-functor $\perf_\fS\to\Perf_\fS$ we get a dg-functor $\nu:R\to\Perf_\fS$, which is a Morita equivalence as we proceed to explain. As explained above, the dg-category $\pretr(R)$ can be identified with the full subcategory of $\perf_\fS$ whose objects are strict perfect complexes of free $\Oo_\fS$-modules. The dg-functor $\nu$ equals the composition
\[R\to \pretr(R)\to \perf_\fS\to \Perf_\fS.\]
The morphism $R\to\pretr(R)$ is a Morita equivalence. Write $\zeta$ for the composite $\pretr(R)\to\Perf_\fS$. The functor $H^0(\zeta)$ is fully faithful and identifies $H^0(\pretr(R))$ with the full subcategory of $\Dperf(\fS)$ whose objects are the strict perfect complexes of free $\Oo_\fS$-modules. Moreover, the functor
\[H^0(\pretr(\zeta))^\sim:H^0(\pretr(R))^\sim\to H^0(\Perf_\fS)^\sim\simeq\Dperf(\fS).\]
is an equivalence since any perfect complex is quasi-isomorphic to a strict perfect complex \cite{tt}*{2.3.1 (d)} and any strict perfect complex is a direct summand of a strict perfect complex of free $R$-modules. Thus $\zeta$ is a Morita equivalence.

\section{Descent properties of Morita invariant and localizing functors}\label{sec:desmor}

\numberwithin{equation}{section}
A functor $E:\dgcat\to\spt$ is \emph{Morita invariant} if it sends Morita equivalences to weak equivalences of spectra. A functor $E:\dgcat\to\spt$ is \emph{localizing} if it sends short exact sequences of dg-categories \cite{kelicm}*{4.6} to homotopy fibration sequences of spectra.

Examples of Morita invariant and localizing functors are non-connective algebraic K-theory $K$ \cite{marcoh}*{3.2.32} as well as Hochschild, cyclic, negative cyclic and periodic cyclic homology, denoted $HH$, $HC$, $HN$ and $HP$ \cite{kelex}*{2.4}. These functors recover the classical invariants for schemes when applied to $\Perf_\fS$. Throughout this paper, cyclic homology and its variants are taken over $\Z$. By a result of Thomason, K-theory has descent for Nisnevich squares and for regular-blowup squares \cite{tt}*{8.1}. By \cite{chsw}*{Theorem 2.10} the same happens to cyclic homology. We recall the following generalisation of these results, which is the particular case for $dg$-categories and schemes over a field of a general result of Tabuada, valid over any commutative ground ring. 

\begin{thm}\cite{tab}*{Theorem 3.1}\label{thm:tabuada}
Let $k$ be a field, $E:\dgcat\to\spt$ a Morita invariant and localizing functor, and $\cal{A}$ a dg-category over $k$. Consider the functor $F:\sch^{\op}\to\spt$ given by $\fS\mapsto E(\Perf_\fS\otk \cal{A})$. Then $F$ has Nisnevich descent. 
\end{thm}

In the same vein we have the following result.

\begin{thm}\label{thm:regularblowups}
Let $k$ be a field, $E:\dgcat\to\spt$ a Morita invariant and localizing functor and $\cal{A}$ a dg-category. Consider the functor $F:\sch^{\op}\to\spt$ given by $\fS\mapsto E(\Perf_\fS\otk \cal{A})$. Then $F$ has descent for regular-blowup squares. 
\end{thm}
\begin{proof}
Let \eqref{square} be a regular blow-up. Assume that $i$ is a regular embedding of pure codimension $d$. We consider the filtration of $\Dperf(\tilde{\fS})$ (resp. $\Dperf(\tilde{\fT})$) by the triangulated subcategories $\Dperf^l(\tilde{\fS})$ (resp. $\Dperf^l(\tilde{\fT})$) defined in \cite{chsw}*{1}. Let $\Perf_{\tilde{\fS}}$ be the pretriangulated dg-category constructed in \ref{subsec:perfect}. For $l=0,\dots,d-1$ let $\Perf _{\tilde{\fS}}^l$ be the full dg-subcategory of $\Perf_{\tilde{\fS}}$ of objects that lie in $\Dperf^l(\tilde{\fS})$. The dg-categories $\Perf_{\tilde{\fS}}^l$ give a filtration of $\Dperf(\tilde{\fS})$ by pretriangulated dg-categories; moreover $H^0(\Perf_{\tilde{\fS}}^l)=\Dperf^l(\tilde{\fS})$. In the same way we construct a filtration of $\Perf_{\tilde{\fT}}$. There is a commutative diagram of dg-categories as follows.
\[\begin{tikzpicture}
\matrix(m)[matrix of math nodes, row sep=1.75em, column sep=1.2em, text height=1.5ex, text depth=0.25ex]
{\Perf_{\fS} & \Perf^0_{\tilde{\fS}} & \Perf^1_{\tilde{\fS}} & \cdots & \Perf^{d-1}_{\tilde{\fS}} = \Perf_{\tilde{\fS}} \\
\Perf_{\fT} & \Perf^0_{\tilde{\fT}} & \Perf^1_{\tilde{\fT}} & \cdots & \Perf^{d-1}_{\tilde{\fT}} = \Perf_{\tilde{\fT}} \\};
\path[->,font=\scriptsize]
(m-1-1) edge node[auto] {$p^*$} (m-1-2)
(m-1-1) edge node[left] {$j^*$} (m-2-1)
(m-2-1) edge node[auto] {$q^*$} (m-2-2)
(m-1-2) edge node[auto] {$j^*$} (m-2-2)
(m-1-3) edge node[auto] {$j^*$} (m-2-3)
(m-1-5) edge node[auto] {$j^*$} (m-2-5);
\path[right hook->]
(m-1-2) edge (m-1-3)
(m-1-3) edge (m-1-4)
(m-1-4) edge (m-1-5)
(m-2-2) edge (m-2-3)
(m-2-3) edge (m-2-4)
(m-2-4) edge (m-2-5);
\end{tikzpicture}
\]
To prove the theorem we must show that after tensoring this diagram with $\cal{A}$ and then applying $E$ the outer square is homotopy cocartesian. The dg-functor $p^*:\Perf_\fS\to \Perf_{\tilde{\fS}}^0$ (resp. $q^*:\Perf_\fT\to \Perf_{\tilde{\fT}}^0$) induces an equivalence upon taking $H^0$ by \cite{chsw}*{Lemma 1.4}; thus it is a Morita equivalence. The dg-functor
\[j^*:\Perf^{l+1}_{\tilde{\fS}}/\Perf^l_{\tilde{\fS}}\to \Perf^{l+1}_{\fT}/\Perf^l_{\fT} \hspace{1em}(l=0,\dots, d-2)\]
induces an equivalence upon taking $H^0$ by \cite{chsw}*{Proposition 1.5}; thus it is a Morita equivalence. The result follows from these observations and the fact that tensoring with $\cal{A}$ preserves Morita equivalences and short exact sequences of dg-categories \cite{drin}*{1.6.3}.
\end{proof}

\section{Equivariant homology theories}\label{sec:eqhom}
\numberwithin{equation}{subsection}
\subsection{Equivariant homology}

Let $G$ be a group. We will write $\sset^G$ for the category of $G$-spaces and equivariant maps. Let $\Or(G)$ be the category whose objects are the cosets $G/H$ and whose morphisms are the homomorphisms of $G$-sets. Let $\grpd$ be the category of small groupoids. Let $\Grpd^G:G-\sets\to\grpd$ be the functor which sends a $G$-set to its transport groupoid. By definition we have $\ob\Grpd^G(U)=U$ and $\hom_{\Grpd^G(U)}(x,y)=\{g\in G/gx=y\}$. Let $F:\grpd\to\spt$ be a functor that sends equivalences of categories to weak equivalences of spectra. As explained in \cite{lr}*{Proposition 157}, $F$ gives rise to a equivariant homology theory putting
\[
H^G(X,F):= \int^{G/H\in\Or(G)}\map_G(G/H,X)\otimes F(\Grpd^G(G/H))
\]
for all $X\in\sset^G$. In the formula above $\otimes$ stands for the simplicial action in the simplicial closed model category of spectra: for $Y\in\sset$ and $E\in\spt$ the spectrum $Y\otimes E$ is obtained by degreewise application of the functor $Y_+\wedge ?$. 
The following lemmas will be used several times.

\begin{lem}\label{lem:weq}
Let $F,F':\grpd\to \spt$ be functors and let $F\to F'$ be an objectwise weak equivalence. Then the map $H^G(X,F)\to H^G(X,F')$ is a weak equivalence of spectra for all $X\in\sset^G$.
\end{lem}
\begin{proof}
It follows from \cite{dl}*{Lemma 4.6}.
\end{proof}

\begin{lem}\label{lem:hocolim}
Let $I$ be a small category and let $F:I\times \grpd\to \spt$ be a functor. Let $E:\grpd\to\spt$ be given by $E(\Grpd)=\hocolim_{i\in I}F(i,\Grpd)$. Let $X\in \sset^G$. Then there is an isomorphism of spectra
\[\hocolim_{i\in I} H^G(X,F(i,?))\simeq H^G(X,E).\]
In particular $H^G(X,?)$ preserves (objectwise) homotopy fibration sequences and homotopy cartesian squares of spectra.
\end{lem}
\begin{proof}
The last assertion follows from the first one using Lemma \ref{lem:weq} because homotopy fibration sequences and homotopy cartesian squares of spectra are homotopy colimits (they are homotopy cofibration sequences and homotopy cocartesian squares respectively). The first assertion follows from Fubini's theorem for coends and from the fact that tensoring with a simplicial set in a simplicial model category commutes with the integral sign since it is a left adjoint.
\end{proof}

\subsection{Equivariant K-theory and cyclic homology}\label{subsec:equikhc}

Let $k$ be a field of characteristic zero. Let $E:\dgcatq\to \spt$ be a Morita invariant functor.

To every scheme $\fS\in\sch$ we associate a functor $E_\fS:\grpd\to\spt$ putting
\[E_\fS(\Grpd):=E(\Perf_\fS\otq \Q[\Grpd]).\]
Because of Morita invariance, $E_\fS$ sends equivalences of categories to weak equivalences of spectra. Hence $E_\fS$ gives rise to an equivariant homology $H^G(?,E(\fS))$. Note that the assignment $\fS\mapsto E_\fS$ is contravariantly functorial, so a morphism of schemes induces a morphism between the corresponding homologies. In this way we obtain equivariant homologies $H^G(?,K(\fS))$, $H^G(?,HC(\fS))$, $H^G(?,HN(\fS))$, and $H^G(?,HP(\fS))$.

To every ring $R$ we associate a functor $E_{R}:\grpd\to\spt$ putting
\[E_{R}(\Grpd):=E({R}[\Grpd]).\]
Because of Morita invariance, $E_{R}$ sends equivalences of categories to weak equivalences of spectra. Then $E_{R}$ gives rise to an equivariant homology $H^G(?,E(R))$. Note that the assignment $R\mapsto E_{R}$ is covariantly functorial.

\begin{rem}\label{rem:agree}
Let $\fS=\Spec R$ be an affine scheme and let $\nu:R\to\Perf_\fS$ be the dg-functor described in \ref{subsec:perfect}. For every groupoid $\Grpd$ the morphism
\[R[\Grpd]=R\otq \Q[\Grpd]\overset{\nu\otimes 1}\longrightarrow \Perf_\fS\otq\Q[\Grpd]\]
is a Morita equivalence and we get a weak equivalence $E_{R}(\Grpd)\weq E_\fS(\Grpd)$ by Morita invariance. By Lemma \ref{lem:weq} we get $H^G(X,E(R))\weq H^G(X,E(\fS))$ for all $X\in \sset^G$.
\end{rem}

\begin{rem}\label{rem:affiso}
Let $\fS=\Spec R$ be an affine scheme. Let $G$ be a group and let $G/H\in \Or G$. Recall that $\hom_{\Grpd^G(G/H)}(H,H)=H$. If we consider $H$ as groupoid with only one object $*$, there is an equivalence of categories $H\to\Grpd^G(G/H)$ which sends $*$ to the coset $H\in G/H$. This equivalence induces a $k$-linear functor $R[H]\to R[\Grpd^G(G/H)]$ which is also a category equivalence, and in particular a quasi-equivalence. By Morita invariance and Remark \ref{rem:agree} we get a weak equivalence 
\begin{equation}\label{eq:affiso}
E(R[H])\weq E_\fS(\Grpd^G(G/H)).
\end{equation}
Note that this weak equivalence is natural in $E$ but not in $G/H$.
\end{rem}

\subsection{Equivariant Chern character and infinitesimal $K$-theory}\label{subsec:equiche}

Let $A$ be a (not necessarily commutative) $k$-algebra. The infinitesimal $K$-theory $K^{\inf}(A)$ is the homotopy fiber of the Chern character $K(A)\to HN(A)$. In this section we define a natural Chern character $K_\fS(\Grpd)\to HN_\fS(\Grpd)$ that coincides with the classical one when $\fS$ is affine and $\Grpd=\Grpd^G(G/H)$ (see Remark \ref{rem:affiso}).

Let $E:\dgcatq\to\spt$ be a functor. For each $\fS\in\sch$ define $E^{\aff}_\fS:\grpd\to\spt$ by $E^{\aff}_\fS(\Grpd):=E_{\Spec \Oo_\fS(\fS)}(\Grpd)$. The morphism $\fS\to\Spec \Oo_\fS(\fS)$ induces a natural transformation $E^{\aff}_\fS\to E_\fS$. Now fix a groupoid $\Grpd$. We have presheaves of spectra $E_?(\Grpd)$ and $E^{\aff}_?(\Grpd)$ on $\sch$. The natural map $E^{\aff}_?(\Grpd)\to E_?(\Grpd)$ is a Zariski-local weak equivalence because both presheaves coincide on affine schemes.

By Morita invariance we have $K_\fS^{\aff}(\Grpd)\weq K(\Oo_\fS(\fS)\otq \Q[\Grpd])$ and $HN_\fS^{\aff}(\Grpd)\weq HN(\Oo_\fS(\fS)\otq \Q[\Grpd])$ naturally in $\fS$ and in $\Grpd$. Then the Chern character for $\Z$-linear categories \cite{corel}*{8.1.6} induces a natural transformation $\ch^{\aff}:K^{\aff}_?(\Grpd)\to HN^{\aff}_?(\Grpd)$. Choose a fibrant replacement functor 
\[
\eta_F:F\mapsto F_{\zar}
\]
for the Zariski-local injective model structure in the category of presheaves of spectra on $\sch$. Consider the following diagram:
\[
\xymatrix{K_?(\Grpd)\ar[d]^{\eta_K} & K^{\aff}_?(\Grpd)\ar[l]\ar[r]^{\ch^{\aff}}\ar[d] & HN^{\aff}_?(\Grpd)\ar[r]\ar[d] & HN_?(\Grpd)\ar[d]^{\eta_{HN}} \\
K_?(\Grpd)_{\zar} & K^{\aff}_?(\Grpd)_{\zar}\ar[l]_{\theta_K}\ar[r]^{\ch^{\aff}_{\zar}} & HN^{\aff}_?(\Grpd)_{\zar}\ar[r]^{\theta_{HN}} & HN_?(\Grpd)_{\zar} \\}
\]
Note that the presheaves $K_?(\Grpd)$ and $HN_?(\Grpd)$ have Zariski descent by Theorem \ref{thm:tabuada}. The morphisms $\theta_K$, $\theta_{HN}$, $\eta_K$ and $\eta_{HN}$ are global weak equivalences beacause they are Zariski-local weak equivalences between presheaves that satisfy Zariski descent. These weak equivalences and the morphism $\ch^{\aff}_{\zar}$ define a Chern character map $\ch:K_\fS(\Grpd)\to HN_\fS(\Grpd)$. It is clear that this map is natural in $\Grpd$ and coincides with the usual Chern character in the affine case.

For $\fS\in\sch$ we define $K^{\inf}_\fS:\grpd\to\spt$ by
\begin{equation}\label{eq:kinf}
K^{\inf}_\fS(\cG):=\hofiber(\ch:K_\fS(\Grpd)\to HN_\fS(\Grpd)).\end{equation}
Because of \eqref{eq:affiso} the Chern character defined above coincides with the usual one in the affine case and there is a weak equivalence
\begin{equation}\label{eq:agreekinf}K^{\inf}_{\Spec R}(\Grpd^G(G/H))\weq K^{\inf}(R[H])\end{equation}
which is natural in $R$. The functor $K^{\inf}_\fS$ sends equivalences of categories to weak equivalences of spectra because both $K_\fS$ and $HN_\fS$ have that property. Thus $K^{\inf}_\fS$ gives rise to an equivariant homology theory $H^G(?,K^{\inf}(\fS))$. By Lemma \ref{lem:hocolim} there is a homotopy fibration sequence
\begin{equation}\label{eq:fseqkinf}H^G(X,K^{\inf}(\fS))\to H^G(X,K(\fS))\to H^G(X,HN(\fS))\end{equation}
for every $X\in \sset^G$ and every $\fS\in\sch$.

\section{Descent results for equivariant homologies}\label{sec:descequi}
\numberwithin{equation}{section}

\begin{lem}\label{lem:nisreg}
Let $X\in \sset^G$ and let $E\in\{K, HC, HN, HP, K^{\inf}\}$. Then the presheaf of spectra $H^G(X,E(?))$ has Nisnevich descent and descent for regular-blowup squares in $\sch$.
\end{lem}
\begin{proof}
Let $E\in\{ K, HC, HN, HP\}$ and let \eqref{square} be a Nisnevich square (resp. a regular blowup square). Theorem \ref{thm:tabuada} (resp. Theorem \ref{thm:regularblowups}) implies that the square
\[
\xymatrix{ E_{\tilde{\fT}}(\Grpd) & E_{\tilde{\fS}}(\Grpd)\ar[l] \\
E_\fT(\Grpd)\ar[u] & E_\fS(\Grpd)\ar[l]\ar[u] \\}
\]
is homotopy cartesian for every groupoid $\Grpd$. It follows from Lemma \ref{lem:hocolim} that $H^G(X,E(?))$ sends \eqref{square} into a homotopy cartesian square of spectra.

In the case of $K^{\inf}$ the assertion follows from the corresponding statements for $K$ and $HN$ and from the homotopy fibration sequence \eqref{eq:fseqkinf}.
\end{proof}

Proposition \ref{prop:cdh} below states that $H^G(X,E(?))$ has cdh-descent on $\sch$ for $E\in\{HP, K^{\inf}\}$. We first recall some definitions and results. Let $\ass$ be the category of rings, and let $\fC\subset\ass$ be a subcategory.

\begin{defi}
A \emph{Milnor square} is a commutative square of rings
\[\xymatrix{ A\ar[r]\ar[d] & B\ar[d] \\
A/I\ar[r] & B/J}\]
in which $I$ is an ideal of $A$, $J$ is an ideal of $B$ and $f$ maps $I$ isomorphically onto $J$.  Let $F:\fC\to\spt$ be a functor. We say that $F$ satisfies \emph{excision} if it sends Milnor squares in $\fC$ to homotopy cartesian squares of spectra.
\end{defi}

The Cuntz-Quillen theorem  (\cite{cq}) establishes excision for $HP$ of $\Q$-algebras. Excision for infinitesimal $K$-theory of $\Q$-algebras was proved in \cite{kabi}.

\begin{defi}
A functor $F:\fC\to\spt$ is \emph{nilinvariant} if for every ring $R\in \fC$ and every nilpotent ideal $I$ of $R$ the map $F(R)\to F(R/I)$ is a weak equivalence.
\end{defi}

Nilinvariance for $HP$ of $\Q$-algebras was proved by Goodwillie in \cite{gooch}*{Theorem II.5.1}. Nilinvariance for $K^{\inf}$ follows from another theorem of Goodwillie \cite{goorel}*{Main Theorem}.

Let $F$ be a presheaf of spectra on $\sch$. Assume that $F$ satisfies Zariski descent. We say that $F$ satisfies \emph{excision} (resp. is \emph{nilinvariant}) if $R\mapsto F(\Spec R)$ satisfies excision (resp. is nilinvariant) on the category of commutative $k$-algebras of finite type. 

We recall the following criterion for $\cdh$-descent.

\begin{thm}[\cite{chsw}*{Theorem 3.12}]\label{thm:chsw}
Let $k$ be a field of characteristic zero. Let $F$ be a presheaf of spectra on $\sch$. Suppose that $F$ satisfies excision, nilinvariance, Nisnevich descent and descent for regular-blowup squares. Then $F$ has cdh-descent.
\end{thm}

\begin{prop}\label{prop:cdh}
Let $k$ be a field of characteristic zero. Let $X\in\sset^G$ and let $E\in\{HP, K^{\inf}\}$. Then the presheaf of spectra $H^G(X,E(?))$ has cdh-descent on $\sch$.
\end{prop}
\begin{proof}
The result will follow from Theorem \ref{thm:chsw}. The presheaf $H^G(X,E(?))$ has Nisnevich descent and descent for regular-blowup squares by Lemma \ref{lem:nisreg}. Let $f:R\to S$ be a morphism of commutative $k$-algebras and let $I$ be an ideal of $R$ which is mapped isomorphically onto an ideal $J$ of $S$. Let $G/H\in\Or G$. Then $f[H]:R[H]\to S[H]$ is a $k$-algebra morphism and $I[H]$ is an ideal of $R[H]$ which is mapped isomorphically onto the ideal $J[H]$ of $S[H]$. The excision results of \cite{cq} and \cite{kabi} imply that the square
\[\xymatrix{E(R[H])\ar[r]\ar[d] & E(S[H])\ar[d] \\
E((R/I)[H])\ar[r] & E((S/J)[H])}\]
is homotopy cartesian. The equivalences \eqref{eq:affiso} and \eqref{eq:agreekinf} imply that the following square is homotopy cartesian:
\[\xymatrix{E_{\Spec R}(\Grpd^G(G/H))\ar[r]\ar[d] & E_{\Spec S}(\Grpd^G(G/H))\ar[d] \\
E_{\Spec R/I}(\Grpd^G(G/H))\ar[r] & E_{\Spec S/J}(\Grpd^G(G/H))}\]
Excision for $H^G(X,E(?))$ follows from Lemma \ref{lem:hocolim}. Nilinvariance for $H^G(X,E(?))$ is proved in a similar way, using the equivalences \eqref{eq:affiso} and \eqref{eq:agreekinf} and the nilinvariance results for $HP$ and $K^{\inf}$ due to Goodwillie and cited above (\cite{gooch}*{Theorem II.5.1} and \cite{goorel}*{Main Theorem}).
\end{proof}

Choose a fibrant replacement functor $F\mapsto F_{\cdh}$ for the cdh-local injective model structure in the category of presheaves of spectra on $\sch$. For a presheaf $F$ let $\Fib F$ be the homotopy fiber of the natural morphism $F\to F_{\cdh}$. Then for every $\fS\in\sch$ we have a homotopy fibration sequence

\begin{equation}\label{seq:fib}
\Fib F(\fS)\to F(\fS)\to F(\fS)_{\cdh}.
\end{equation}

\begin{thm}[cf. \cite{chwvorst}*{Theorem 1.6}]\label{thm:hgfc}
Let $G$ be a group and $X$ a $G$-space. Also let $k$ be a field of characteristic zero  and $\fS\in\sch$. Then there is a homotopy fibration sequence
\[
\Fib H^G(X,HC(\fS))[-1]\to H^G(X,K(\fS))\to H^G(X,K(\fS))_\cdh .
\]
\end{thm}
\begin{proof} It follows from Proposition \ref{prop:cdh} by the argument of \cite{chwvorst}*{Theorem 1.6}. 
\end{proof}

\section{Isomorphism conjectures}\label{sec:ic}

Recall from the Introduction the definition of the strong isomorphism conjecture for a cuadruple $(G,\Fam,E, R)$. By a result of L\"{u}ck and Reich \cite{lrdet}*{Theorem 1.7} the strong isomorphism conjecture with coefficients in $HC(R)$ holds for any group $G$ and any family $\Fam$ containing all cyclic subgroups of $G$. We shall prove a variant of L\"uck-Reich's result with coefficients in a scheme. We need some notation. Let $k$ be a field of characteristic zero and let $\fS\in\sch$. Also let $G$ be a group and $f:X\to Y$ an equivariant map of $G$-spaces. Then $f$ induces a map of homotopy fibration sequences
\begin{equation}\label{map:hcseq}
\xymatrix{
\Fib H^G(X,HC(\fS))\ar[r]\ar[d] & H^G(X,HC(\fS))\ar[d]\ar[r]& H^G(X,HC(\fS))_{\cdh}\ar[d] \\
\Fib H^G(Y,HC(\fS))\ar[r] & H^G(Y,HC(\fS))\ar[r] & H^G(Y,HC(\fS))_{\cdh} \\
}
\end{equation}

\begin{prop}\label{prop:ichc}
Let $\Fam$ be a family of subgroups containing all the cyclic subgroups of $G$. Assume that $f:X\to Y$ is a $(G,\Fam)$-equivalence. Then the vertical maps in \eqref{map:hcseq} are weak equivalences. 
\end{prop}
\begin{proof}
By \cite{corel}*{Proposition 7.6} the vertical map in the middle of \eqref{map:hcseq} is a weak equivalence for affine $\fS$. Hence 
\begin{equation}\label{map:xyhc}
H^G(X,HC(?))\to H^G(Y,HC(?))
\end{equation}
is a Zariski local weak equivalence; by Lemma \ref{lem:nisreg} it is a global weak equivalence. This shows that the vertical map in the middle of \eqref{map:hcseq} is a weak equivalence. Similarly, because the map \eqref{map:xyhc} is a Zariski local equivalence, it is also a $\cdh$-local equivalence.
Hence it induces a global equivalence $H^G(X,HC(?))_\cdh\to H^G(Y,HC(?))_\cdh$. This proves that the rightmost vertical map of \eqref{map:hcseq} is a weak equivalence. By the five lemma, it follows that also the leftmost vertical map is a weak equivalence. This completes the proof. 
\end{proof}

We are now ready to prove our main theorem.

\begin{thm}\label{thm:main}
Let $\Fam$ be a family of subgroups of $G$ that contains all the cyclic subgroups. Let $k$ be a a field of characteristic zero and let $f:X\to Y$ be a $(G,\Fam)$-equivalence. Suppose that $H^G(f,K(R))$ is a weak equivalence for every commutative smooth $k$-algebra $R$. Then $H^G(f,K(R))$ is a weak equivalence for every commutative $k$-algebra $R$. In particular, if the (strong) isomorphism conjecture for $(G,\Fam, K, R)$ holds for every commutative smooth $k$-algebra $R$, then it holds for every commutative $k$-algebra $R$.
\end{thm}
\begin{proof}
Because $\Fam\supset\cyc$, the map $\Fib H^G(f,HC(\fS))$ is a weak equivalence for all $\fS\in\sch$, by Proposition \ref{prop:ichc}.
Because $H^G(f,K(R))$ is a weak equivalence for smooth $R$, $H^G(f,K(?))$ is a $\cdh$-local weak equivalence. Hence $H^G(f,K(\fS))_\cdh$ is a weak equivalence
for every $\fS\in\sch$. By Theorem \ref{thm:hgfc} and what we have just proven, $H^G(f,K(\fS))$ is a weak equivalence for every $\fS\in\sch$. In particular  
$H^G(f,K(R))$ is a weak equivalence for $R$ commutative and of finite type over $k$. Because $H^G(f,K(?))$ commutes with filtering colimits up to homotopy, it follows that $H^G(f,K(R))$ is a weak equivalence for every commutative $k$-algebra $R$. 
\end{proof}

\begin{coro}\label{coro:main}
Let $G$ be a group. If $G$ satisfies the $K$-theoretic Farrell-Jones conjecture with coefficients in every commutative smooth $\Q$-algebra $R$ then it also satisfies the Farrell-Jones conjecture with coefficients in any commutative $\Q$-algebra.
\end{coro}

\end{document}